\documentclass[11pt]{amsart}

\usepackage{amsmath}
\usepackage{amssymb}
\usepackage{amsthm}
\usepackage{graphicx}
\usepackage[utf8]{inputenc}
\usepackage[inline]{enumitem}
\usepackage{ amssymb }
\usepackage{hyperref}
 \usepackage{mathtools}
 \usepackage{caption}
 \usepackage{ mathrsfs }

\newtheorem{theorem}{Theorem}[section]
\newtheorem{corollary}[theorem]{Corollary}
\newtheorem{lemma}[theorem]{Lemma}
\newtheorem{prop}[theorem]{Proposition}
\theoremstyle{definition}
\newtheorem{definition}[theorem]{Definition}
\newtheorem{example}[theorem]{Example}
\newtheorem{remark}[theorem]{Remark}
\newtheorem{claim}[theorem]{Claim}
\newtheorem{notation}[theorem]{Notation}
\newtheorem{question}[theorem]{Question}
\newtheorem{conjecture}[theorem]{Conjecture}

\numberwithin{equation}{section}

\newcommand{\vertiii}[1]{{\left\vert\kern-0.25ex\left\vert\kern-0.25ex\left\vert #1 
\right\vert\kern-0.25ex\right\vert\kern-0.25ex\right\vert}}

\newcommand{\R}{\mathbb{R}}
\newcommand{\C}{\mathbb{C}}
\newcommand\restr[2]{{
  \left.\kern-\nulldelimiterspace 
  #1 
  \vphantom{\big|} 
  \right|_{#2} 
  }}

\newenvironment{innerproof}
 {\proof}
 {\endproof}

\title[Decomposition of planar Lipschitz quotient mappings]{On the structural decomposition of planar Lipschitz quotient mappings}

\author[R. Hutchins]{Ricky Hutchins}
\address[R. Hutchins]{School of Mathematics, University of Birmingham, Edgbaston, Birmingham, B15 2TT, U.K.}
\email{{\tt rxh595@bham.ac.uk}}

\author[O. Maleva]{Olga Maleva}
\address[O. Maleva]{School of Mathematics, University of Birmingham, Edgbaston, Birmingham, B15 2TT, U.K.}
\email{\tt o.maleva@bham.ac.uk}

\keywords{Lipschitz quotient, strongly co-Lipschitz}

\subjclass[2010]{46B20, 46T20}

\begin{document}

\begin{abstract}
We show that for each fixed non-constant complex polynomial $P$ of the plane there exists a homeomorphism $h$ such that $P\circ h$ is a Lipschitz quotient mapping. This corrects errors in the construction given earlier in \cite{JLPS}. Further we introduce a stronger notion of pointwise co-Lipschitzness and characterise its equivalence to the standard pointwise definition whilst also highlighting its relevance to a long-standing conjecture concerning Lipschitz quotient mappings $\R^n\to\R^n$, $n\geq 3$. 
\end{abstract}

\maketitle

\section{Introduction}\label{introduction}
The motivation for this paper follows from a desire to understand how much planar Lipschitz quotient mappings are tied to the underlying structure of such mappings as discovered in \cite{JLPS}, see also Theorem~\ref{Lipschitz quotient decomposes into a complex polynomial and a homeomorphim} below. For a pair of metric spaces $X$ and $Y$ mappings $f:X\to Y$ are called Lipschitz quotient mappings provided they are Lipschitz and additionally satisfy a `dual' property of being co-Lipschitz. Namely, a mapping $f$ is Lipschitz quotient if there exist constants $0<c\leq L<+\infty$ such that 
    \begin{equation*} 
        B^Y_{cr}\left(f(x)\right)\stackrel{(1)}{\subseteq} f\left(B^X_r(x)\right)\stackrel{(2)}{\subseteq} B^Y_{Lr}\left(f(x)\right)
    \end{equation*}
for any $x\in X$ and all $r>0$. Here $B_s^Z(z)$ denotes the open ball of radius $s>0$ centred at $z\in Z$ where $Z=X,Y$. If only inclusion (2) is satisfied for each $x\in X$ and for every $r>0$, we say $f$ is $\left(L-\right)$Lipschitz. Similarly, if only inclusion (1) holds for each $x\in X$ and for every $r>0$, we say $f$ is $\left(c-\right)$co-Lipschitz.

If $f$ is a Lipschitz mapping we define $\text{Lip}(f)$, the Lipschitz constant of $f$, to be the infimum over all such $L>0$ for which inclusion (2) holds. Similarly, if $f$ is a continuous co-Lipschitz mapping we define $\text{co-Lip}(f)$, the co-Lipschitz constant of $f$, to be the supremum over all such $c>0$ for which inclusion (1) holds.  We remark that to guarantee the supremum for the co-Lipschitz constant exists we impose the continuity restriction since by assuming the axiom of choice there exist functions, for example from $\R$ to $\R$, which are surjective to $\R$ on every non-empty open subset, cf. \cite{everywhere}.

Pointwise notions of co- and Lipschitz mappings have been considered in different texts, for example \cite{Csornyei, MV}. Here if (2) is satisfied for a fixed $x\in X$ and all $r<r_0$, for some $r_0=r_0(x)>0$, then we say $f$ is pointwise $L$-Lipschitz at $x$. We define the notion of pointwise co-Lipschitzness similarly. Other local notions of co-Lipschitzness have also been considered, for example in \cite{MalevaMathe}. Another local notion we are going to use is local injectivity. We say a mapping $f:X\to Y$ between two metric spaces is locally injective at $x\in X$ if there exists $r>0$ such that the restriction of $f$ to $B_r^X(x)$ is injective. 

Co-Lipschitz mappings were first introduced in \cite{Gromov, James, Whyburn} but first systematically studied in \cite{BJLPS,JLPS}. The results in \cite{BJLPS} support the intuitive notion that Lipschitz quotient mappings are non-linear analogues for linear quotient mappings between Banach spaces. When considering linear quotient mappings $X\to Y$ the point preimage of each $x\in X$ is an affine subspace of $X$ with dimension $d:=\dim(X)-\dim(Y)$. In \cite{OlgaPaper1} it is shown that provided the constants $c$ and $L$ are sufficiently close then point preimages, under Lipschitz quotient mappings, cannot be $(d+1)$-dimensional. However, if no condition is imposed on the constants, it is shown in \cite{Csornyei} that there exist Lipschitz quotient mappings $\R^3\to\R^2$ which collapse a subset containing a 2-dimensional plane to a singular point. 

Such a result is not possible for planar Lipschitz quotient mappings, no matter how far the constants $c$ and $L$ are. It is shown in \cite{JLPS} that such mappings have a very specific structure:

\begin{theorem}[{\cite[Theorem~2.8(i)]{JLPS}}]\label{Lipschitz quotient decomposes into a complex polynomial and a homeomorphim}
Suppose $f:\C\to\C$ is a Lipschitz quotient mapping. Then $f=P\circ h$, where $h:\C\to\C$ is a homeomorphism and $P$ is a complex polynomial of one complex variable. 
\end{theorem}

Note that in Theorem~\ref{Lipschitz quotient decomposes into a complex polynomial and a homeomorphim} we did not specify the norms associated to the domain and co-domain. This is justified since passing to an equivalent norm preserves whether a mapping is a Lipschitz quotient; hence in the finite dimensional setting there is no need to specify the norm considered. We also highlight here that the original statement of Theorem~\ref{Lipschitz quotient decomposes into a complex polynomial and a homeomorphim} in \cite{JLPS} is given for Lipschitz quotient mappings from $\R^2$ to itself. The restatement in Theorem~\ref{Lipschitz quotient decomposes into a complex polynomial and a homeomorphim} follows due to the natural identification of $\R^2$ with $\C$, which is required in the definition of the polynomial $P$ in any case. 

Surprisingly little is known for higher dimensional analogues of Theorem~\ref{Lipschitz quotient decomposes into a complex polynomial and a homeomorphim}. It is not even known if Lipschitz quotient mappings $\R^n\to \R^n$, $n\geq 3$ are discrete, see Conjecture~\ref{conjecture regarding discreteness of Lip quo}. 

In light of Theorem~\ref{Lipschitz quotient decomposes into a complex polynomial and a homeomorphim}, and in the search of converses to such a statement, the authors of \cite{JLPS} pose questions regarding the uniqueness of the homeomorphism $h$ obtained from the decomposition of a Lipschitz quotient mapping and whether a converse statement to Theorem~\ref{Lipschitz quotient decomposes into a complex polynomial and a homeomorphim} holds also. It is shown that, up to a linear transformation, the homeomorphism obtained via the decomposition of a Lipschitz quotient mapping is unique, see \cite[p.~22]{JLPS}. 

In connection to the structural decomposition of planar Lipschitz quotient mappings, we ask the following questions concerning converse statements to Theorem~\ref{Lipschitz quotient decomposes into a complex polynomial and a homeomorphim}. 

\begin{question}\label{question: positive converses}
\begin{enumerate}[label=(\alph*)]
    \item \label{fixed h, Lip quo} Can every planar homeomorphism $h:\C\to\C$ be obtained via a decomposition of a Lipschitz quotient mapping? In other words, is it true that for every homeomorphism $h:\C\to\C$ there exists a non-constant complex polynomial $P$ such that $P\circ h$ is a Lipschitz quotient mapping? 
    \item \label{fixed P, Lip quo} Can every non-constant complex polynomial $P$ be obtained via a decomposition of a planar Lipschitz quotient mapping? In other words, is it true that for every non-constant polynomial $P$ there exists a homeomorphism $h:\C\to\C$ such that $P\circ h$ is a Lipschitz quotient mapping?
\end{enumerate}
\end{question}
We begin by considering Question~\ref{question: positive converses}~\ref{fixed h, Lip quo}. We provide a planar homeomorphism $h$ such that $P\circ h$ is not Lipschitz quotient for every non-constant complex polynomial $P$. Indeed, consider the homeomorphism $h:\C\to\C$ given by $h(z)=|z|^2e^{i\arg(z)}$. Observe that $P\circ h$ is not Lipschitz for every non-constant complex polynomial $P$. This follows simply as
    \begin{equation*}
       \lim\limits_{R\to +\infty} \dfrac{\left|P\circ h(R)-P\circ h(0)\right|}{R}= +\infty.
    \end{equation*}
The main motivation of this paper is to consider Question~\ref{question: positive converses}~\ref{fixed P, Lip quo}, as the authors of \cite{JLPS} do. The authors claim to answer this in \cite[Proposition~2.9]{JLPS} in the positive, and provide a sketch proof of the following statement. 

\begin{prop}\label{main result}
Let $P$ be a non-constant polynomial in one complex variable with complex coefficients. Then there exists a homeomorphism $h$ of the plane such that $f=P\circ h$ is a Lipschitz quotient mapping. 
\end{prop}

However, as we show in Section~\ref{construction of Lip quo}, the construction of their mapping $h$ is not in fact a homeomorphism of the plane. In this paper we prove Proposition~\ref{main result}. To do so we follow the framework provided in \cite{JLPS} but correct oversights in the original sketch and in doing so introduce a stronger (pointwise) notion of co-Lipschitzness, namely strongly co-Lipschitz.

With this new notion, we pose a question regarding the existence of Lipschitz quotient mappings $\R^n\to\R^n$, $n\geq 3$ which are strongly co-Lipschitz at some point. We explain the logical equivalence between this question and the long-standing conjecture of \cite{JLPS} whether such mappings are necessarily discrete. Moreover, with this new notion, we consider the following question. 

\begin{question}\label{question: negative converse}
For a fixed homeomorphism $h:\C\to\C$ does there exist a non-constant complex polynomial $P$ such that $P\circ h$ is not a Lipschitz quotient mapping?
\end{question}

We answer Question~\ref{question: negative converse} in the positive in Lemma~\ref{corr about fixed homeo, not lip qup}. 

\section{Preliminaries}
Throughout this paper, for a metric space $X$ and $S\subseteq X$, $\text{Int}(S)$ denotes the topological interior of $S$ and $\partial S$ represents the boundary of $S$.

\begin{notation}\label{definition of S(P')}
For any $z\in\C\setminus\{0\}$ we take $\arg(z)\in \left(-\pi,\pi\right]$ to denote the \textit{principal argument} of $z$. Further, for any $a>0$, $b\in\R$ we define $|z|^ae^{ib\arg(z)}=0$ when $z=0$. 

For any non-constant complex polynomial $P$ in one complex variable and $a>0$ we define the closed set 
        \begin{equation}\label{V_beta r definition}
            V^P_{a}=\bigcup\limits_{z_j\in S(P')}\overline{B}_{a}(z_j),
        \end{equation}
where $P'$ is the derivative of $P$ and $S(P')=\left\{z\in\C:P'(z)=0\right\}$.
\end{notation}

We now state properties of particular functions which are important in the judicious choose of $r>0$ which we are making in Claim~\ref{choice of r}. First, let $P$ be a fixed non-constant complex polynomial of one complex variable, $P'$ be its derivative and $z_j\in S(P')$. Of course if $P$ is non-zero and linear, then $S(P')=\emptyset$. Define the polynomial
    \begin{equation}\label{equation for Q_j}
        Q_j(z):=\dfrac{P(z)-P(z_j)}{(z-z_j)^{m_j}},
    \end{equation} 
where $m_j\geq 1$ is the multiplicity of $z_j$ as a root of the polynomial $P(z)-P(z_j)$. Note, for future reference, that $P(z)=(z-z_j)^{m_j}Q_j(z)+P(z_j)$. Further, by the maximality of $m_j$,
    \begin{equation}\label{Q_j(z_j) is not zero}
        Q_j(z_j)\neq 0.
    \end{equation}
We define the expansion of the polynomial $Q_j$ about $z_j$ by 
    \begin{equation}\label{taylor expansion of Q_j about z_j}
        Q_j(z)=\sum\limits_{l=0}^{n-m_j}c_{l,j}(z-z_j)^l
    \end{equation}
where $n=\text{deg}(P)$ and $c_{l,j}\in\C$. Thus \eqref{Q_j(z_j) is not zero} implies  ${c_{0,j}=Q_j(z_j)\neq 0}$ for each $j$ such that $z_j\in S(P')$.

We now define a function which proves useful in the construction of the Lipschitz quotient mapping in Section~\ref{construction of Lip quo}. For each $m\geq 1$, let $A_m\subseteq \C\times\C$ be defined by 
    \begin{equation*}
        A_m:=\left\{(z,w):|z|e^{im\arg(z)}\neq |w|e^{im\arg(z)}\right\}\cup\left\{(w,w):w\in\C\setminus\{0\}\right\}. 
    \end{equation*}
Now, for each $m\geq 1$ and $l\in \{1,\dots,m\}$ we define the mapping $\Phi_{l,m}:A_m\to\C$ by 
    \begin{equation}\label{definition of function Phi}
        \Phi_{l,m}(z,w)=\begin{cases}
            \dfrac{|z|^{\frac{l+m}{m}}e^{i(l+m)\arg(z)}-|w|^{\frac{l+m}{m}}e^{i(l+m)\arg(w)}}{|z|e^{im\arg(z)}-|w|e^{im\arg(w)}},\quad&\text{if $z\neq w$};\\
            \text{ }\\
            \dfrac{l+m}{m}|w|^{\frac{l}{m}}e^{il\arg(w)},\quad&\text{if $z=w$}.\\
           
        \end{cases}
    \end{equation}
    
\begin{lemma} \label{Phi is continuous}
    Let $m\geq 1$ and $1\leq l\leq m$. For each $w\in \C\setminus\{0\}$, there exists $\rho>0$ such that $B_{\rho}(w)\times\{w\}\subseteq A_m$ and 
        \begin{equation*}
            \lim\limits_{\substack{z\to w\\ z\in B_{\rho}(w)}}\Phi_{l,m}(z,w)=\Phi_{l,m}(w,w). 
        \end{equation*}
\end{lemma}

\begin{proof}
Note for $w\in\C\setminus\{0\}$ fixed that there exist finitely many points $z\in\C$ such that $(z,w)\not\in A_m$; namely this happens exactly when $z\neq w$ but $|z|=|w|$ and $e^{im\arg(z)}=e^{im\arg(w)}$. Hence, there exists $\rho>0$ such that $B_{\rho}(w)\times\{w\}\subseteq A_m$. 

If $z\in B_{\rho}(w)\setminus\{w\}$, then $\Phi_{l,m}(z,w)=(g(f(z))-g(f(w)))/(f(z)-f(w))$ where $f,g:\C\to\C$ are given by $f(z)=|z|e^{im\arg(z)}$ and $g(z)=z^{(l+m)/m}$. As $w$ is fixed, $f$ is continuous at $w$ and $g$ is differentiable at $f(w)$, observe that 
    \begin{align*}
        \lim\limits_{\substack{z\to w\\ z\in B_{\rho}(w)}}\Phi_{l,m}(z,w)=\lim\limits_{\substack{z\to w\\ z\in B_{\rho}(w)}}\dfrac{g(f(z))-g(f(w))}{f(z)-f(w)}=g'(f(w))=\Phi_{l,m}(w,w). 
    \end{align*}
\end{proof}

\begin{corollary} \label{Phi is bounded}
Let $m\geq 1$ and $1\leq l\leq m$. For each $w\in\C\setminus\{0\}$ and $\varepsilon>0$ there exists $\rho>0$ such that $B_{\rho}(w)\times\{w\}\subseteq A_m$ and whenever $z\in B_{\rho}(w)$, 
    \begin{equation}\label{properties of Phi following from uniform continuity}
      |\Phi_{l,m}\left(z,w\right)|<\varepsilon+\left|\Phi_{l,m}(w,w)\right|. 
    \end{equation}
\end{corollary}

The following result concerns planar mappings which have the inherent structure of a Lipschitz quotient mapping. The below identifies a finite set $E$ such that mappings of the form $P\circ h$ are locally injective on $\C\setminus E$. In the following proof $\text{card}(S)$ represents the cardinality of the set $S$.  

\begin{prop}\label{reworded cristina}
Let $f:\C\to\C$ be a mapping such that $f=P\circ h$ where $P$ is a non-constant complex polynomial in one complex variable and $h:\C\to\C$ is a homeomorphism. Then there exists a finite subset $E\subseteq \C$ such that $f$ is locally injective at each $x\in \C\setminus E$. Moreover, $E=h^{-1}\left(S(P')\right)$. 
\end{prop}

\begin{proof}
Fix $x_0\in\C$ such that $h(x_0)\not\in S(P')$. We claim $f$ is locally injective at $x_0$. Since $P'(h(x_0))\neq 0$, by \cite[Theorem~7.5]{manifolds}, there exists an open neighbourhood $N_{h(x_0)}$ of $h(x_0)$ such that $\restr{P}{N_{h(x_0)}}$ is injective. Therefore $f=P\circ h$ is injective on the open neighbourhood $G=h^{-1}(N_{h(x_0)})$ of $x_0$.

As this holds for any $x_0\in \C$ such that $h(x_0)\not\in S(P')$, $f$ is locally injective outside of $E=h^{-1}(S(P'))$. Since $P'$ is a non-zero polynomial, $\text{card}(S(P'))\leq \text{deg}(P)-1$. As $h$ is bijective, $\text{card}(E)=\text{card}(S(P'))$.
\end{proof}

We state a couple of standard results regarding Lipschitz mappings. 

\begin{lemma}\label{continuous and lip on set, then lip on closure}
Let $X$, $Y$ be metric spaces, $A\subseteq X$ dense in $X$ and $L>0$. If $f:X\to Y$ is a continuous mapping such that $\restr{f}{A}$ is $L$-Lipschitz, then $f$ is $L$-Lipschitz. 
\end{lemma}

The following lemma ensures that a mapping which is pointwise Lipschitz everywhere, with a uniform constant, is necessarily Lipschitz, with the same constant. However, for this we need the linear structure induced by normed spaces. 

\begin{lemma}\label{pointwise Lip on convex open implies restriction is lipschitz}
Let $X,Y$ be normed spaces, $U\subseteq X$ be open and convex and $L>0$. If $f:X\to Y$ is pointwise $L$-Lipschitz at each $x\in U$, then $\restr{f}{U}$ is $L$-Lipschitz. 
\end{lemma}

Recall \cite[Section~4]{Csornyei} and \cite[Lemma~2.3]{MV} which introduce a result analogous to Lemma~\ref{pointwise Lip on convex open implies restriction is lipschitz} for co-Lipschitz mappings in the case $U=X=Y=\C$.

\begin{lemma}\label{everywhere local co-Lip}
Let $c>0$. If $f:(\C,\|\cdot\|)\to(\C,\vertiii{\cdot})$ is continuous and is pointwise $c$-co-Lipschitz at every $x\in \C$, then $f$ is (globally) $c$-co-Lipschitz. 
\end{lemma}

Homeomorphisms between two metric spaces preserve pointwise co- and Lipschitzness of such mappings and their inverses in the following manner. 

\begin{lemma}\label{inverse is lipschitz}
Let $X$ and $Y$ be metric spaces, $h:X\to Y$ be a homeomorphism, $x_0\in X$ and $c>0$. Then $h$ is pointwise $c$-co-Lipschitz at $x_0$ if and only if $h^{-1}$ is pointwise $(1/c)$-Lipschitz at $h(x_0)$. 
\end{lemma}

\begin{proof}
If $h$ is pointwise $c$-co-Lipschitz at $x_0$ there exists $r_0>0$ such that $B_{cr}^Y(h(x_0))\subseteq h\left(B_r^X(x_0)\right)$ for each $r\in(0,r_0)$. Therefore
    \begin{equation*} 
        h^{-1}\left(B_{cr}^Y(h(x_0))\right)\subseteq h^{-1}\left(h\left(B_r^X(x_0)\right)\right)=B_r^X(x_0)= B_r^X\left(h^{-1}\left(h\left(x_0\right)\right)\right)
    \end{equation*} 
for each $r\in(0,r_0)$. Hence $h^{-1}$ is pointwise $(1/c)$-Lipschitz at $h(x_0)$. The reverse direction follows similarly. 
\end{proof}

The traditional examples of planar Lipschitz quotient mappings $f_n$, as defined in Lemma~\ref{archetypal case}, possess sharp constants, in the sense that the ratios of constants $c/L$ for such mappings are maximal, cf. \cite[Theorem~2]{OlgaPaper1}. 

\begin{lemma}\label{archetypal case}
For each $n\in\mathbb{N}$ define $f_{n}:(\C,|\cdot|)\to(\C,|\cdot|)$ to be given by $f_{n}(z)=|z|e^{in\arg(z)}$. Then $f_n$ is a Lipschitz quotient mapping; namely $f_n$ is $n$-Lipschitz and $1$-co-Lipschitz with respect to the Euclidean norm. 
\end{lemma}

\begin{remark}
We highlight here that in Corollary~\ref{standard is strongly co-Lipschitz}, which we prove later, we show that $f_n$ satisfy properties which are stronger than 1-co-Lipschitzness. 
\end{remark}

The following lemma concerns the Lipschitz property of variants of the standard Lipschitz quotient mappings $f_n$ introduced in Lemma~\ref{archetypal case}.

\begin{lemma}\label{k/n general function lipschitz}
Let $n\in\mathbb{N}$ and $k\in\{1,\dots,n-1\}$. For each $\varepsilon>0$ there exists $D=D\left(\varepsilon,k,n\right)>0$ such that $g_{k,n}:\C\setminus B_{D}(0)\to\C$ defined by $g_{k,n}(z)=|z|^{k/n}e^{ik\arg(z)}$ is $\varepsilon$-Lipschitz on $\C\setminus B_{D}(0)$. 
\end{lemma}

\begin{proof}
Fix $\varepsilon>0$. Define $f_k(z)=|z|e^{ik\arg(z)}$ for $z\in\C$ as in Lemma~\ref{archetypal case}. Further, define $h_k(t)=t^{k/n}$ for $t>0$. Let $T>0$ be such that $h_k$ is $(\varepsilon/2)$-Lipschitz on $[T,+\infty)$ and let $R>0$ be such that $\frac{k+1}{R^{1-k/n}}<\frac{\varepsilon}{2}$. Define $D:=\max\left\{T,R\right\}$. Fix $z_1,z_2\in \C\setminus B_{D}(0)$. Now
    \begin{equation}\label{splitting up of g_{k,n}}
        \begin{aligned}
        |g_{k,n}(z_1)&-g_{k,n}(z_2)|\leq\\
        &\left|g_{k,n}(z_1)-|z_2|^{k/n}e^{ik\arg(z_1)}\right|+\left|z_2\right|^{k/n}\left|e^{ik\arg(z_1)}-e^{ik\arg(z_2)}\right|.
        \end{aligned}
    \end{equation}
As $|z_1|,|z_2|\geq D\geq T$ and as $h_k$ is $(\varepsilon/2)$-Lipschitz on $[T,+\infty)$, 
    \begin{equation}\label{first term estimate for g_{k,n} lipschitz proof}
        \begin{aligned}
            \left|g_{k,n}(z_1)-|z_2|^{k/n}e^{ik\arg(z_1)}\right|=\left|h_k(|z_1|)-h_k(|z_2|)\right|&\leq \dfrac{\varepsilon}{2}\bigg||z_1|-|z_2|\bigg|
            \\
            &\leq \dfrac{\varepsilon}{2}\left|z_1-z_2\right|.
        \end{aligned}
    \end{equation}
Further, since $|z_2|\geq D\geq R$, 
        \begin{align*}
            &|z_2|^{k/n}\left|e^{ik\arg(z_1)}-e^{ik\arg(z_2)}\right|\\
            &\leq \left||z_2|^{k/n}-|z_1|\cdot|z_2|^{\frac{k}{n}-1}\right|+\left||z_1|\cdot|z_2|^{\frac{k}{n}-1}e^{ik\arg(z_1)}-|z_2|^{\frac{k}{n}}e^{ik\arg(z_2)}\right|\\
            &=\dfrac{1}{|z_2|^{1-\frac{k}{n}}}\bigg(\bigg||z_1|-|z_2|\bigg|+\left|f_k(z_1)-f_k(z_2)\right|\bigg)\leq \dfrac{\varepsilon}{2}\left|z_1-z_2\right|,
        \end{align*}
where the last inequality follows by the choice of $R>0$ and Lemma~\ref{archetypal case}. Substituting this and \eqref{first term estimate for g_{k,n} lipschitz proof} into \eqref{splitting up of g_{k,n}} we obtain 
    \begin{equation*} 
        \left|g_{k,n}(z_1)-g_{k,n}(z_2)\right|\leq \varepsilon |z_1-z_2|.
\end{equation*} 
By the arbitrariness of $z_1,z_2\in \C\setminus B_{D}(0)$ we conclude the required Lipschitzness of $g_{k,n}$. 
\end{proof}

We now introduce a quick lemma regarding the composition of pointwise co-Lipschitz functions.

\begin{lemma}\label{composition of locally co-lip is locally co-lip}
Suppose $X$, $Y$ and $Z$ are metric spaces and $f:X\to Y$, $g:Y\to Z$ are functions. Suppose $f$ is pointwise $a$-co-Lipschitz at $x\in X$ and $g$ is pointwise $b$-co-Lipschitz at $f(x)\in Y$ for some constants $a,b>0$. Then $g\circ f$ is pointwise $\left(ab\right)$-co-Lipschitz at $x$. 
\end{lemma}

\begin{proof}
As $f$ is pointwise $a$-co-Lipschitz at $x\in X$, there exists $\rho_f>0$ such that $f(B_r^X(x))\supseteq B_{ar}^Y(f(x))$ for each $r\in(0,\rho_f)$. Similarly, there exists $\rho_g>0$ such that $g(B_r^Y(f(x)))\supseteq B_{br}^Z(g(f(x)))$ for each $r\in (0,\rho_g)$. Define $\rho:=\min(\rho_f,\rho_g/a)$. Then, for each $r\in(0,\rho)$, 
    \begin{equation*}
        (g\circ f)\left(B_r^X(x)\right)\supseteq g\left(B_{ar}^Y(f(x))\right)\supseteq B_{abr}^Z\left((g\circ f)(x)\right).
    \end{equation*}
Hence, $g\circ f$ is pointwise $(ab)$-co-Lipschitz at $x\in X$.
\end{proof}

The next lemma provides a sufficient property for a mapping between metric spaces to be pointwise co-Lipschitz at a given point. To be able to conveniently refer to this property, we first give the following definition. 

\begin{definition}\label{definition for strongly co-Lipschitz}
Suppose $(X,d_X)$ and $(Y,d_Y)$ are metric spaces and $c>0$. We say a function $f:X\to Y$ is \textit{strongly $c$-co-Lipschitz} at $x_0\in X$ if there exists $\rho>0$ such that:
    \begin{enumerate}[label=(\roman*)]
        \item \label{interior condition for strongly definition}$f(x_0)\in\text{Int}\left(f\left(B_{\rho}^X(x_0)\right)\right)$;
        \item \label{inequality for strongly} $d_Y(f(x),f(x_0))\geq cd_X(x,x_0)$ for all $x\in B_{\rho}^X(x_0)$.
    \end{enumerate}
If we do not need to specify $c$, we shall simply write $f$ is strongly co-Lipschitz at $x_0$. 
\end{definition}

\begin{lemma}\label{nicer co-Lipschitz condition for non-open maps}
Let $(X,d_X)$ and $(Y,d_Y)$ be metric spaces and $c>0$. If $f:X\to Y$ is strongly $c$-co-Lipschitz at $x_0\in X$, then $f$ is pointwise $c$-co-Lipschitz at $x_0$. 
\end{lemma}

\begin{proof}
Let $\rho>0$ be as in Definition~\ref{definition for strongly co-Lipschitz}. By property~\ref{interior condition for strongly definition} of Definition~\ref{definition for strongly co-Lipschitz} there exists a positive constant $R<\rho$ such that 
    \begin{equation}\label{interior condition for the fixed point in the pointwise co-Lip lemma}
        B^Y_R\left(f(x_0)\right)\subseteq \text{Int}\left(f\left(B^X_{\rho}(x_0)\right)\right)\subseteq f\left(B^X_{\rho}(x_0)\right).
    \end{equation}
Define $r:=\frac{R}{2c}>0$, let $0<s<r$ and fix $y\in B^Y_{c s}\left(f(x_0)\right)$. By the choice of $r$, note $cs<cr<R$. Thus \eqref{interior condition for the fixed point in the pointwise co-Lip lemma} implies $y\in f\left(B^X_{\rho}(x_0)\right)$. Hence there exists $x\in B^X_{\rho}(x_0)$ such that $y=f(x)$. We claim $x\in B^X_s(x_0)$ follows by \ref{inequality for strongly} of Definition~\ref{definition for strongly co-Lipschitz}. Indeed, since $x\in B^X_{\rho}(x_0)$ and $y\in B^Y_{c s}(f(x_0))$,
    \begin{equation*} 
        cs>d_Y(y,f(x_0))=d_Y(f(x),f(x_0))\geq c d_X(x,x_0),
    \end{equation*}
so $x\in B^X_s(x_0)$. Therefore $y=f(x)\in f\left(B^X_s(x_0)\right)$ and since $y\in B^Y_{c s}\left(f(x_0)\right)$ was arbitrary we deduce 
$B^Y_{c s}\left(f(x_0)\right)\subseteq f\left(B^X_s(x_0)\right)$. Finally, since $s\in (0,r)$ was arbitrary we conclude $f$ is pointwise $c$-co-Lipschitz at $x_0$. 
\end{proof}

\begin{corollary}\label{nicer co-lip condition for open maps}
Let $(X,d_X)$, $(Y,d_Y)$ be metric spaces. Suppose $f:X\to Y$ is an open map, $x_0\in X$ and there exist positive constants $c$ and $r_0$ such that $d_Y(f(x),f(x_0))\geq cd_X(x,x_0)$ for each $x\in B^X_{r_0}(x_0)$. Then $f$ is pointwise $c$-co-Lipschitz at $x_0$. 
\end{corollary}

\begin{remark}\label{remark: strongly co-Lip in open subsets}
When proving pointwise or strong co-Lipschitzness of mappings defined in Section~\ref{construction of Lip quo}, we will often consider $X$ to be an open subset of $\C$. In such cases, instead of $B_r^X(x)$, we will consider balls centred at $x\in X$ and open in the Euclidean metric. To be able to use the definition of a co-Lipschitz mapping or Definition~\ref{definition for strongly co-Lipschitz} and subsequent results about strongly co-Lipschitz mappings, it is enough to ensure $r$ is sufficiently small so that the Euclidean ball of radius $r$ around $x$ coincides with $B_r^X(x)$. 
\end{remark}

\begin{remark}
Using the notion introduced in Definition~\ref{definition for strongly co-Lipschitz}, the following implication follows by Lemma~\ref{nicer co-Lipschitz condition for non-open maps}:
    \begin{equation}\label{strongly co-Lip implies pointwise}
        \text{strongly $c$-co-Lipschitz at }x_0\implies \text{pointwise $c$-co-Lipschitz at }x_0.
    \end{equation}
\end{remark}

One may naturally ask the question of whether a reverse implication holds. In Lemma~\ref{pointwise implies strongly sufficient to show the inequality} below, we show that only property~\ref{inequality for strongly} of Definition~\ref{definition for strongly co-Lipschitz} needs to be verified for a pointwise co-Lipschitz mapping to be strongly co-Lipschitz. 

\begin{lemma}\label{pointwise implies strongly sufficient to show the inequality}
Let $(X,d_X)$ and $(Y,d_Y)$ be metric spaces, $f:X\to Y$, $x_0\in X$ and $c>0$. Suppose $f$ is pointwise $c$-co-Lipschitz at $x_0$. If there exists $\rho_0>0$ such that $d_Y(f(x_0),f(z))\geq cd_X(x_0,z)$ for each $z\in B_{\rho_0}^X(x_0)$, then $f$ is strongly $c$-co-Lipschitz at $x_0$. 
\end{lemma}

\begin{proof}
It is enough to prove (i) of Definition~\ref{definition for strongly co-Lipschitz} is satisfied for some $0<\rho<\rho_0$. Indeed, as $f$ is pointwise $c$-co-Lipschitz at $x_0$, there exists a positive $r_0$ such that $f(B_r^X(x_0))\supseteq B_{cr}^Y(f(x_0))$ for each $r\in(0,r_0)$. Define $\rho:=\frac{1}{2}\min(r_0,\rho_0)$. Then $f(x_0)\in B_{c\rho}^Y(f(x_0))\subseteq f(B_{\rho}^X(x_0))$. Hence as $B_{c\rho}^Y(f(x_0))$ is open, we deduce (i) is satisfied. Thus $f$ is strongly $c$-co-Lipschitz at $x_0$. 
\end{proof}

The reverse implication of \eqref{strongly co-Lip implies pointwise} can easily be seen in the case when the function is locally injective, as we show in the following lemma. 

\begin{lemma}\label{pointwise implies strongly}
Let $(X,d_X)$, $(Y,d_Y)$ be metric spaces, $x_0\in X$ and $c>0$. Suppose a mapping $f:X\to Y$ is both pointwise $c$-co-Lipschitz and locally injective at $x_0$. Then $f$ is strongly $c$-co-Lipschitz at $x_0$. 
\end{lemma}

\begin{proof}
Since $f$ is pointwise $c$-co-Lipschitz at $x_0$, by definition, there exists $r_0>0$ such that
    \begin{equation} \label{co-Lipschitz condition in the proof of equivalence}
        B^Y_{cr}\left(f(x_0)\right)\subseteq f\left(B^X_r(x_0)\right) \quad\text{for each }r\in (0,r_0).
    \end{equation}
Since $f$ is locally injective at $x_0$ there exists $r_1>0$ such that $\restr{f}{B^X_{r_1}(x_0)}$ is injective. Define $\rho:=\frac{1}{2}\min\left(r_0,r_1\right)$. Recall Lemma~\ref{pointwise implies strongly sufficient to show the inequality}. Thus it suffices to show
    \begin{equation} \label{strongly co-Lipschitz inequality in the proof of equivalence}
        d_Y(f(x),f(x_0))\geq cd_X(x,x_0)\quad\text{for all }x\in B_\rho^X(x_0).
    \end{equation}
This is trivially satisfied for $x=x_0$. Suppose, for a contradiction, that \eqref{strongly co-Lipschitz inequality in the proof of equivalence} is not satisfied, i.e. there exists $x\in B_\rho^X(x_0)\setminus\{x_0\}$ such that $d_Y(f(x),f(x_0))<cd_X(x,x_0)$. Define $r:=d_X(x,x_0)$, so $0<r<\rho<r_0$. Hence, $f(x)\in B_{cr}^Y(f(x_0))\subseteq f\left(B_r^X(x_0)\right)$ where the inclusion follows by \eqref{co-Lipschitz condition in the proof of equivalence}. Therefore, as $\restr{f}{B_{\rho}^X(x_0)}$ is injective, $x\in B_{\rho}^X(x_0)$ and $B_r^X(x_0)\subseteq B_{\rho}^X(x_0)$, it follows $x\in B_r^X(x_0)$ and so $r=d_X(x,x_0)<r$, providing contradiction. Hence \eqref{strongly co-Lipschitz inequality in the proof of equivalence} is satisfied.
\end{proof}

\begin{corollary}\label{equivalence of strongly co-Lipschitz and pointwise co-Lipschitz}
Suppose $X$ and $Y$ are metric spaces, $f:X\to Y$ is a mapping which is locally injective at $x_0\in X$ and $c>0$. Then 
    \begin{equation*}
        f\text{ is strongly }c\text{-co-Lipschitz at }x_0\text{ } \Longleftrightarrow \text{ }f\text{ is pointwise }c\text{-co-Lipschitz at }x_0.
    \end{equation*}
\end{corollary}

\begin{remark}
We highlight the relevance of Corollary~\ref{equivalence of strongly co-Lipschitz and pointwise co-Lipschitz} in the context of mappings with the inherent structure of planar Lipschitz quotient mappings. Indeed Proposition~\ref{reworded cristina} identifies at which points of the plane a composition $P\circ h$ of a polynomial $P$ and a homeomorphism $h$ is locally injective, hence where the notions of strongly co-Lipschitzness and pointwise co-Lipschitzness agree. In Corollary~\ref{Lip quo is strongly co-Lipschitz} below, we show that these two notions automatically agree everywhere for any Lipschitz quotient mapping. However, as mentioned in Section~\ref{introduction}, not all mappings with this underlying structure $P\circ h$ are Lipschitz quotient.
\end{remark}

Further, we are able to show the equivalence between the two notions of pointwise co-Lipschitz and strongly co-Lipschitz for discrete co-Lipschitz mappings. To see this we follow the method presented in \cite[p. 2091]{BLD}. Let us first recall the definition of a discrete mapping. 

\begin{definition}
Let $X,Y$ be topological spaces and $S\subseteq X$. We say:
\begin{itemize} 
    \item $S$ is a \textit{discrete set} if for each $x\in S$ there exists a neighbourhood $U$ of $x$ such that $U\cap S=\left\{x\right\}$;
    \item  $f:X\to Y$ is a \textit{discrete mapping} if $f^{-1}\left(y\right)$ is a discrete set for each $y\in Y$.
\end{itemize}
\end{definition}

\begin{lemma}\label{mapping discrete then pointwise and strongly are equivalent}
Suppose $(X,d_X)$, $(Y,d_Y)$ are metric spaces and $f:X\to Y$ is a discrete $c$-co-Lipschitz mapping for some $c>0$. Then $f$ is strongly $c$-co-Lipschitz at each $x\in X$. 
\end{lemma}

\begin{proof}
Fix $x\in X$ and define $\mathcal{A}_x=f^{-1}\left(f(x)\right)$. Since $f$ is a discrete mapping there exists $r_0>0$ such that $B\left(x,2r_0\right)\cap\mathcal{A}_x=\left\{x\right\}$. Fix $z\in  B^X_{r_0}(x)\setminus\{x\}$ and let $r:=d_X(z,x)$. Then $B^X_r(z)\cap \mathcal{A}_x=\emptyset$, so $f(x)\not\in f\left(B^X_r(z)\right)$. Since $f$ is $c$-co-Lipschitz, $f\left(B^X_r(z)\right)\supseteq B^Y_{cr}\left(f(z)\right)$. As $f(x)\not\in f(B_r^X(z))$ this implies $d_Y(f(x),f(z))\geq cr=cd_X(x,z)$. 

Observe that $d_Y(f(x),f(z))\geq cd_X(x,z)$ is trivially satisfied when $z=x$. Therefore, by Lemma~\ref{pointwise implies strongly sufficient to show the inequality}, we conclude $f$ is strongly $c$-co-Lipschitz at $x$. 
\end{proof}

We highlight that Lemma~\ref{pointwise implies strongly} and Lemma~\ref{mapping discrete then pointwise and strongly are equivalent} are the strongest possible, in the sense that there exist Lipschitz quotient mappings which are $1$-co-Lipschitz but not locally injective, not discrete and are not strongly co-Lipschitz at any point. We show this in the following example. 

\begin{example}\label{discretness of co-Lipschitz mapping is important}
Let $n,k\geq 1$ be integers and $f:\R^{n+k}\to \R^n$ be the standard projection, where both spaces are equipped with the Euclidean norm. Then $f$ is $1$-Lipschitz and $1$-co-Lipschitz. This trivially follows since $f\left(B_r\left(x\right)\right)=B_r\left(f(x)\right)$ for each $r>0$ and $x\in\R^{n+k}$. Further, it is clear that $f$ is not discrete. Moreover, $f$ is neither injective nor strongly $c$-co-Lipschitz, for any $c>0$, at any $x_0\in\R^{n+k}$ as $f^{-1}(x_0)$ is a $k$-dimensional hyperplane. 
\end{example}

Using Lemma~\ref{mapping discrete then pointwise and strongly are equivalent}, we deduce the following two corollaries. First we show that planar Lipschitz quotient mappings, or any continuous co-Lipschitz planar mappings, are necessarily strongly co-Lipschitz at every point.

\begin{corollary}\label{Lip quo is strongly co-Lipschitz}
Suppose $f:\C\to\C$ is a continuous $c$-co-Lipschitz mapping for some $c>0$. Then $f$ is strongly $c$-co-Lipschitz at each $x\in \C$.
\end{corollary}

\begin{proof}
By \cite[Proposition~4.3]{BJLPS}, or equivalently \cite[Proposition~2.1]{JLPS}, $f$ is discrete and so Lemma~\ref{mapping discrete then pointwise and strongly are equivalent} yields the result. 
\end{proof}

\begin{corollary}\label{standard is strongly co-Lipschitz}
For every $n\in\mathbb{N}$ let the function $f_n:\C\to\C$ be defined by $f_n(z)=|z|e^{in\arg(z)}$ as in Lemma~\ref{archetypal case}. Then $f_n$ is strongly 1-co-Lipschitz at each $z\in\C$. 
\end{corollary}

Following Corollary~\ref{Lip quo is strongly co-Lipschitz} one may ask the following question. 

\begin{question}\label{question regarding existence of discrete lip quo which is not strongly}
For $n\geq 3$ do there exist Lipschitz quotient mappings $f:\R^n\to \R^n$ which are not strongly co-Lipschitz at some $x_0\in\R^n$?
\end{question}

We highlight the logical equivalence between Question~\ref{question regarding existence of discrete lip quo which is not strongly} and a long-standing conjecture from \cite[p.~1096]{BJLPS}. Namely: 

\begin{conjecture}\label{conjecture regarding discreteness of Lip quo}
Suppose $n\geq 3$ and $f:\R^n\to\R^n$ is a Lipschitz quotient mapping. Then $f$ is a discrete mapping. 
\end{conjecture}

First we note that a positive answer to Conjecture~\ref{conjecture regarding discreteness of Lip quo} implies, via an application of Lemma~\ref{mapping discrete then pointwise and strongly are equivalent}, that every Lipschitz quotient mapping $f:\R^n\to \R^n$, $n\geq 3$ is strongly $c$-co-Lipschitz everywhere, where $c=\text{co-Lip}(f)$, providing a negative answer to Question~\ref{question regarding existence of discrete lip quo which is not strongly}. 

Conversely a negative answer to Question~\ref{question regarding existence of discrete lip quo which is not strongly}, i.e. every Lipschitz quotient mapping $f:\R^n\to\R^n$ is strongly co-Lipschitz everywhere, implies Conjecture~\ref{conjecture regarding discreteness of Lip quo}. This implication is proved in the following simple lemma.

\begin{lemma}\label{strongly co-Lip on fibre implies fibre is discrete}
Let $\left(X,d_X\right)$, $\left(Y,d_Y\right)$ be metric spaces and $y\in Y$. If $f:X\to Y$ is strongly co-Lipschitz at each element of $f^{-1}(y)$, then $f^{-1}(y)$ is a discrete set. 

In particular, if $f$ is strongly co-Lipschitz at every $x\in X$, then $f$ is a discrete mapping. 
\end{lemma}

\begin{proof}
To show $f^{-1}(y)$ is discrete we require to show for each $x\in f^{-1}(y)$ that there exists a neighbourhood $U_x$ of $x$ such that $U_x\cap f^{-1}(y)=\left\{x\right\}$. Fix $x\in f^{-1}(y)$. Since $f$ is strongly co-Lipschitz at $x$, there exist positive constants $c_x$ and $\rho_x$ such that 
    \begin{equation}\label{equation in lemma for strongly on fibre implies fibre is discrete}
        d_Y\left(f(w),f(x)\right)\geq c_x d_X\left(w,x\right)\quad\text{for each }w\in B_{\rho_x}^X(x). 
    \end{equation}
Define $U_x:=B_{\rho_x}^X(x)$. Let $z\in U_x\cap f^{-1}(y)$. Then since $z\in B_{\rho_x}^X(x)$ and $f(z)=y$, by \eqref{equation in lemma for strongly on fibre implies fibre is discrete} it follows that $0=d_Y(f(z),f(x))\geq c_x d_X(z,x)$. Thus $z=x$ since $c_x>0$ and so $U_x\cap f^{-1}(y)=\left\{x\right\}$. Since $x\in f^{-1}(y)$ was arbitrary, we conclude $f^{-1}(y)$ is a discrete set. 
\end{proof}

With the introduction of the notion of strong co-Lipschitzness, we are in a position to answer Question~\ref{question: negative converse} affirmatively. Formally, we prove the following. 

\begin{lemma}\label{corr about fixed homeo, not lip qup}
Let $h:\C\to\C$ be a homeomorphism. Then there exists a complex polynomial $P$ in one complex variable such that $P\circ h$ is not Lipschitz quotient. 
\end{lemma}

Naturally Lemma~\ref{corr about fixed homeo, not lip qup} is a consequence that squaring Lipschitz quotient mappings of the plane never produces a Lipschitz mapping, also. We prove this in the following lemma.

\begin{lemma}\label{square is not a Lip quo}
Suppose $f:\C\to\C$ is a Lipschitz quotient mapping. Then $g(z)=(f(z))^2$ is not Lipschitz. 
\end{lemma}

\begin{proof}
Suppose $f$ is $c_f$-co-Lipschitz and $L_f$-Lipschitz and, for a contradiction, suppose $g$ is Lipschitz. Let us assume, without loss of generality, that both $g/f$ are Lipschitz/Lipschitz quotient with respect to the Euclidean norm. Now \cite[Theorem~2.8~(1)]{MV} provides the existence of a positive constant $R$ such that 
    \begin{equation} \label{lower bound for f olga's result}
        \left|f(x)\right|\geq c_f\left(\left|x\right|-M\right),
    \end{equation}
whenever $|x|>R$. Here $M:=\max\left\{|z|:f(z)=0\right\}$. Let $L_g>0$ be such that $g$ is $L_g$-Lipschitz and fix $z_0\in \C$ such that $|z_0|> R+M+L_g/(2c_f^2)$. Since $f$ is strongly $c_f$-co-Lipschitz at $z_0$, by Corollary~\ref{Lip quo is strongly co-Lipschitz}, there exists $r_0\in (0,1)$ such that $\left|f(z_0)-f(w)\right|\geq c_f\left|z_0-w\right|$ for all $w\in B_{r_0}(z_0)$. As $g$ is $L_g$-Lipschitz, 
    \begin{align*}
        c_f\left|z_0-w\right|\left|f(z_0)+f(w)\right|\leq\left|\left(f(z_0)\right)^2-\left(f(w)\right)^2\right|&=\left|g(z_0)-g(w)\right|\\
        &\leq L_g\left|z_0-w\right|,
    \end{align*}
for all $w\in B_{r_0}(z_0)$. Hence, for any $w\in B_{r_0}(z_0)\setminus \left\{z_0\right\}$, $\left|f(z_0)+f(w)\right|\leq L_g/c_f$. Thus, by the continuity of $f$, $|f(z_0)|\leq L_g/(2c_f)$. However, by our choice of $z_0$ and \eqref{lower bound for f olga's result}, $\left|f(z_0)\right|> L_g/(2c_f)$, providing contradiction. Hence $g$ is not Lipschitz. 
\end{proof}

\section{Construction of the Lipschitz quotient mapping}\label{construction of Lip quo}
Recall the function $h:\C\to \C$ given in \cite[Proposition~2.9]{JLPS} (for some large $R>0$): 
    \begin{equation}\label{function given by JLPS}
        h(z)=\begin{cases}
                z,\quad&\text{if $|z|\leq R$},\\
                \left(\dfrac{2R-|z|}{R}|z|+\dfrac{|z|-R}{R}|z|^{1/n}\right)e^{i\text{arg}(z)},\quad&\text{if $R\leq |z|\leq 2R$},\\
                |z|^{1/n}e^{i\text{arg}(z)},\quad&\text{if $|z|\geq 2R$}.
             \end{cases}
    \end{equation}
The authors of \cite{JLPS} claim first this is a homeomorphism from $\C$ to itself and go on to provide a sketch for a proof of Proposition~\ref{main result}. However it is clear that $h$ is not injective by observing that, for $R>2^{1/(n-1)}$ if $n>1$, the curve $\partial B_{2R}(0)$ is mapped under $h$ inside the open ball $B_R(0)$ where the mapping remains fixed. Further, the authors introduce an amendment to the function $h$ which may further provide points at which $h$ is not injective. They describe how to change the function $h$ defined by \eqref{function given by JLPS} on a finite collection of open balls. However they neglect the fact the prescribed radii of these balls are potentially very small and hence will require a `scaling' to ensure the function is necessarily injective, as indicated by the $r^{1-(1/m_j)}$ term in \eqref{h_2 equation}. Finally, the authors state the co-Lipschitzness of the function $h$ outside of the union of these balls, but do not verify the co-Lipschitzness on their boundaries, which is intricate.

Below we give a correct construction, for a fixed polynomial $P$, of a homeomorphism $h$ of the plane to itself such that 
$P\circ h$ is a Lipschitz quotient mapping. The proof of Proposition~\ref{main result} will be split into many claims, which verify the pointwise co- and Lipschitz property of the required functions, and remarks, which utilise earlier lemmata to conclude co- and Lipschitzness on specific regions. To highlight the end of the proof of a claim we use the symbol $\diamondsuit$, whereas the end of the proof of the proposition is highlighted by the usual $\square$.

\begin{proof}[Proof of Proposition~\ref{main result}]
Fix $n\in\mathbb{N}$. We may assume without loss of generality that $P$ is a monic polynomial of degree $n$. Indeed if $P$ is not monic, let $a\neq 0$ denote the leading coefficient of $P$. One can apply the present Proposition to the monic polynomial $Q:=P/a$ to find the homeomorphism $h$ such that $f(z)=(Q\circ h)(z)$ is a Lipschitz quotient mapping. Then $(P\circ h)(z)=af(z)$ is a Lipschitz quotient mapping.

Therefore, assume $P(z)=z^n+a_{n-1}z^{n-1}+\dots+a_1z+a_0$. If $n=1$ define $h(z):=z$ and then $f(z)=\left(P\circ h\right)(z)=z+a_0$ is $1$-co-Lipschitz and $1$-Lipschitz.

Suppose $n\geq 2$. The structure of the proof is as follows: we begin by defining a homeomorphism $h_1$ of the plane, let $F_1=P\circ h_1$ and show that $F_1$ is Lipschitz on $\C$ and pointwise co-Lipschitz on $\C$ with the exception of a small neighbourhood $W$ of finitely many points. Namely, $W$ contains a neighbourhood of the set of roots of the polynomial $P'$, the derivative of $P$. We use this to show $F_1$ is strongly co-Lipschitz at each $z\in\C\setminus V$, where $V\supseteq W$. We then proceed by defining an amended homeomorphism $h_2$ which coincides with $h_1$ everywhere outside of $V$, define the new function $F_2=P\circ h_2$ and prove $F_2$ is pointwise co- and Lipschitz at the remaining points. Let us introduce some notation which will be important in the construction.

\begin{notation}\label{D_k notation and choice of R}
    If $a_k\neq 0$ let $D_k=D\left(1/(2n|a_k|),k,n\right)$ be provided by Lemma~\ref{k/n general function lipschitz}, such that $g_{k,n}(z)=|z|^{k/n}e^{ik\arg(z)}$ is $1/(2n|a_k|)$-Lipschitz on $\mathbb{R}^2\setminus B_{D_k}(0)$; otherwise if $a_k=0$, let $D_k=0$.
\end{notation}

Let $R>1$ be such that
        \begin{enumerate}[label=(\alph*)]
            \item \label{roots of P' inside ball of radius R} the roots of the derivative $P'$ lie inside the open ball of radius $R$ centred at the origin;
            \item \label{R is sufficiently large concerning lipschitz} $R\geq \max\left\{D_k:0\leq k\leq n-1\right\}$.
        \end{enumerate} 
Define $h_1:\C\to\C$ by
        \begin{equation*}    
            h_1(z)=\phi(|z|)e^{i\arg(z)},
        \end{equation*} 
    where 
        \begin{equation*} 
            \phi(t)=\begin{cases}
                    t^{1/n},\quad&\text{if $t\geq 2^nR^n$};\\
                    \left(\dfrac{t-R}{2^nR^{n-1}-1}+R\right),\quad&\text{if $R\leq t\leq 2^nR^n$};\\
                    t,\quad&\text{if $0\leq t\leq R$}.
                \end{cases}
    \end{equation*}
Since $\phi:[0,+\infty)\to[0,+\infty)$ is a continuous, piecewise $C^{\infty}$ strictly increasing function, $h_1$ is bijective and continuous. Further we note $h_1^{-1}(z)=\phi^{-1}(|z|)e^{i\arg(z)}$ which is continuous. Hence $h_1$ is indeed a homeomorphism of $\C$ to itself. Finally, let $U_j:=B_{2^nR^n+j}(0)$ for $j=1,2$. Define $F_1=P\circ h_1$. 

\begin{claim} \label{F_1 is lip on U_2}
    $F_1$ is Lipschitz on $U_2$.
\end{claim}

\begin{innerproof}
We first show that $h_1$ is Lipschitz on $U_2$. Note that $h_1$ is pointwise 1-Lipschitz at each $z_0\in B_R(0)$, since if $r>0$ if sufficiently small such that $B_r(z_0)\subseteq B_R(0)$, then $h_1\left(B_r\left(z_0\right)\right)=B_r\left(z_0\right)=B_r\left(h_1(z_0)\right)$. 

To see that $h_1$ is pointwise Lipschitz at each $z_0\in U_2\setminus \overline{B}_{R/2}(0)$, first note that $\phi$ is Lipschitz on $[R/2,2^nR^n+2]$. Moreover observe that $e^{i\arg(z)}=z/|z|$ is Lipschitz on $\C\setminus B_{R/2}(0)$, as if $z,w\in \C\setminus B_{R/2}(0)$, then 
    \begin{equation*}
        \left|\dfrac{z}{|z|}-\dfrac{w}{|w|}\right|\leq \dfrac{1}{|z|\cdot|w|}\left(|w|\cdot|z-w|+|w|\cdot\bigg||w|-|z|\bigg|\right)\leq \dfrac{4}{R}|z-w|. 
    \end{equation*}
Thus, $h_1(z)=\phi(|z|)e^{i\arg(z)}$ is the product of two bounded Lipschitz functions on the bounded domain $A=\{z\in \C:R/2\leq |z|\leq 2^nR^n+2\}$. Therefore, $\restr{h_1}{A}$ is $L$-Lipschitz for some $L>0$. In particular, we conclude that $h_1$ is pointwise $L$-Lipschitz at each $z\in U_2\setminus\overline{B}_{R/2}(0)$. 

Therefore Lemma~\ref{pointwise Lip on convex open implies restriction is lipschitz} implies $h_1$ is $\max\left(1,L\right)$-Lipschitz on the convex, open set $U_2$. Now, $F_1=P\circ h_1$ is the composition of $P$, a polynomial, which is Lipschitz on the bounded set $h_1(U_2)$ and $h_1$, which is Lipschitz on $U_2$. Therefore, $F_1$ is Lipschitz on $U_2$.
\end{innerproof}

\begin{claim} \label{F_1 is Lip outside of U_1}
    $F_1$ is Lipschitz on $\C\setminus \overline{U_1}$.
\end{claim}

\begin{innerproof}
To see $F_1$ is Lipschitz outside of $\overline{U_1}$ note for $z\not\in \overline{U_1}$ that $F_1(z)$ takes the specific form
    \begin{equation} \label{f specific form}
        F_1(z)=a_0+f_n(z)+\sum\limits_{k=1}^{n-1}a_kg_{k,n}(z),
    \end{equation}
where $f_n$ is defined as in Lemma~\ref{archetypal case} and $g_{k,n}$ as in Lemma~\ref{k/n general function lipschitz} for each $k\in\{1,\dots,n-1\}$.

Hence, as $f_n$ is $n$-Lipschitz on $\C$ by Lemma~\ref{archetypal case}, to show $F_1$ is Lipschitz on $\C\setminus\overline{U_1}$ it suffices to show for each $k\in\{1,\dots,n-1\}$ that $a_kg_{k,n}$ is Lipschitz on $\C\setminus \overline{U_1}$; this follows by Lemma~\ref{k/n general function lipschitz} and the choice of $R$ and $D_k$ in Notation~\ref{D_k notation and choice of R}~(b). Hence $F_1$ is Lipschitz on $\C\setminus \overline{U_1}$. 
\end{innerproof}

\begin{remark}\label{F_1 is Lipschitz on C}
Recall by Claims~\ref{F_1 is lip on U_2}, \ref{F_1 is Lip outside of U_1} that $F_1$ is Lipschitz on both $\C\setminus\overline{U_1}$ and $U_2$. Therefore Lemma~\ref{pointwise Lip on convex open implies restriction is lipschitz} yields that there exists $L_1>0$ such that $F_1$ is $L_1$-Lipschitz on $\C$.
\end{remark}

\begin{claim}\label{choice of r} Recall \eqref{V_beta r definition}-\eqref{taylor expansion of Q_j about z_j} from Notation~\ref{definition of S(P')} and the choice of $R$ from Notation~\ref{D_k notation and choice of R}. There exists $r\in\left(0,1\right)$ such that:
        \begin{enumerate}[label=(\roman*)]
            \item \label{property i of r}the balls $\overline{B}_{2r}(z_j)$ around roots $z_j\in S(P')$ of $P'$, are pairwise disjoint;
            \item  \label{property ii of r} $V^P_{2r}\subseteq B_{R}(0)$;
            \item  \label{property iii of r} $r\leq\min\limits_{j:z_j\in S(P')}\varepsilon_j^{m_j}$, where for each $z_j\in S(P')$ we define $\varepsilon_j>0$ by 
                \[
                    \varepsilon_j:=
                        \begin{cases}
                            \dfrac{|Q_j(z_j)|}{2(1+n)\sum\limits_{k=1}^{n-m_j}|c_{k,j}|},\quad&\text{if $n>m_j$ and $\sum\limits_{k=1}^{n-m_j}|c_{k,j}|\neq 0$},\\
                            1,\quad&\text{otherwise.}\\
                        \end{cases}
                \]
            \item\label{property iv of r} $|Q_j(z_j)|/2\leq|Q_j(y)|\leq 2|Q_j(z_j)|$ for each $y\in B_r(z_j)$ such that $z_j\in S(P')$. 
        \end{enumerate}
\end{claim}

\begin{innerproof}
Property \ref{property i of r} is easy to satisfy as there are only finitely many distinct roots in $S(P')$. Next, property \ref{property ii of r} is satisfied for sufficiently small $r>0$ since $S(P')\subseteq B_R(0)$ and $B_R(0)$ is open. Property \ref{property iii of r} follows naturally by \eqref{Q_j(z_j) is not zero} since each $\varepsilon_j$ is positive and there are only finitely many of these terms. Finally, it is possible to satisfy property~\ref{property iv of r} since each polynomial $Q_j$ is continuous on $\C$ and $Q_j(z_j)\neq 0$ by \eqref{Q_j(z_j) is not zero}. 
\end{innerproof}

For the rest of the proof of Proposition~\ref{main result}, we fix $r\in (0,1)$ provided by Claim~\ref{choice of r}. Recall \eqref{V_beta r definition}, and define the closed sets $W$ and $V$ to be the following: 
    \begin{equation}\label{definition of the closed sets for the construction proof}
        W=V^P_{r/2}, \quad V=V^P_r.
    \end{equation}

\begin{claim}\label{F_1 co-Lip in U_2 outside W}
    There exists $c_0>0$ such that $F_1$ is pointwise $c_0$-co-Lipschitz at each $z\in U_2\setminus W$.
\end{claim}

\begin{innerproof}
We first show that there exist positive constants $L$ and $\xi$ such that $h_1$ is pointwise $(1/L)$-co-Lipschitz at each $z\in U_2$ and the polynomial $P$ is pointwise $\xi$-co-Lipschitz at each $z\in h_1\left(U_2\setminus W\right)$. Then we appeal to Lemma~\ref{composition of locally co-lip is locally co-lip} to conclude that $F_1$ is pointwise $c_0:=\left(\frac{\xi}{L}\right)$-co-Lipschitz at each $z\in U_2\setminus W$. 

By arguing similarly to the proof of Claim~\ref{F_1 is lip on U_2}, namely as $h_1^{-1}(z)=\phi^{-1}(|z|)e^{i\arg(z)}$ is the product of two bounded Lipschitz functions, there exists $L>0$ such that $h_1^{-1}$ is pointwise $L$-Lipschitz at $h_1(z)$ for each $z\in U_2$. Thus Lemma~\ref{inverse is lipschitz} and the arbitrariness of $z\in U_2$ implies $h_1$ is pointwise $(1/L)$-co-Lipschitz at each $z\in U_2$. 

Observe by Claim~\ref{choice of r}~\ref{property ii of r} that $S(P')\subseteq W\subseteq B_R(0)$. Therefore, as $h_1$ is the identity on $B_R(0)$ and since $\left|h_1(z)\right|\geq R$ for $|z|\geq R$, we conclude that $\overline{h_1(U_2\setminus W)}\cap S(P')=\emptyset$. As $P'$ is a polynomial, hence continuous, $|P'|$ assumes its minimal value $2\xi>0$ on the compact set $\overline{h_1(U_2\setminus W)}$. In particular for each $z\in h_1(U_2\setminus W)$ note $P'(z)\neq 0$ and thus, by \cite[Theorem~7.5]{manifolds}, there exist open neighbourhoods $N_{P(z)}\subseteq F_1\left(U_2\setminus W\right)$ and $N_z\subseteq h_1\left(U_2\setminus W\right)$ of $P(z)$ and $z$ respectively such that $P:N_z\to N_{P(z)}$ is a continuous bijective open mapping, hence a homeomorphism. Further, $(P^{-1})'(P(z))=1/P'(z)$. Therefore for each $z\in h_1\left(U_2\setminus W\right)$ it follows that $|(P^{-1})'(P(z))|\leq 1/(2\xi)$. Hence $P^{-1}$ is pointwise $\frac{1}{\xi}$-Lipschitz at $P(z)$. By Lemma~\ref{inverse is lipschitz} and Remark~\ref{remark: strongly co-Lip in open subsets} we hence conclude $P$ is pointwise $\xi$-co-Lipschitz at $z$ since $P:N_z\to N_{P(z)}$ is a homeomorphism, $N_z$ and $N_{P(z)}$ are open subsets of $\C$ and $z\in N_z$. We conclude $P$ is pointwise $\xi$-co-Lipschitz at each $z\in h_1\left(U_2\setminus W\right)$. 

Now $h_1$ is pointwise $\frac{1}{L}$-co-Lipschitz at each $z\in U_2\setminus W$ and $P$ is pointwise $\xi$-co-Lipschitz at each $h_1(z)\in h_1\left(U_2\setminus W\right)$. Therefore by Lemma~\ref{composition of locally co-lip is locally co-lip} we conclude $F_1$ is pointwise $c_0$-co-Lipschitz at each $z\in U_2\setminus W$ where $c_0=\xi/L>0$. 
\end{innerproof}

\begin{remark}\label{remark about strongly co-Lipschitz outside of Int(V)}
Since $\left(U_2\setminus W\right)\cap h_1^{-1}(S(P'))=\emptyset$, by Proposition~\ref{reworded cristina}, $F_1$ is locally injective at each $z\in U_2\setminus W$. Further, $U_2\setminus W$ is open. Therefore Remark~\ref{remark: strongly co-Lip in open subsets}, Corollary~\ref{equivalence of strongly co-Lipschitz and pointwise co-Lipschitz} and Claim~\ref{F_1 co-Lip in U_2 outside W} imply $F_1$ is strongly $c_0$-co-Lipschitz at each $z\in U_2\setminus W$. In particular, for each $z\in U_2\setminus \text{Int}(V)$ there exists $\rho=\rho(z)>0$ such that $B_{\rho}(z)\subseteq U_2\setminus W$ and 
    \begin{equation}\label{F_1 is strongly co-Lipschitz outside of Int(V)}
        \left|F_1(z)-F_1(x)\right|\geq c_0\left|z-x\right|\quad\text{ for all }x\in B_\rho(z). 
    \end{equation}
\end{remark}

\begin{claim} \label{F_1 is pointwise co-Lip outside U_1 bar}
    $F_1$ is $\frac{1}{2}$-pointwise co-Lipschitz at each $z\in \C\setminus \overline{U_1}$.
\end{claim}

\begin{innerproof}
Fix any $z_0\in\C\setminus \overline{U_1}$. Recall $F_1=P\circ h_1$ where $P$ is a non-constant polynomial of one variable, so is an open map, and $h_1$ is a homeomorphism. Therefore $F_1$ is open. By Corollary~\ref{nicer co-lip condition for open maps} and Remark~\ref{remark: strongly co-Lip in open subsets}, as $\C\setminus\overline{U_1}$ is open, to check that $F_1$ is pointwise $(1/2)$-co-Lipschitz at $z_0$, it is enough to verify property~\ref{inequality for strongly} of Definition~\ref{definition for strongly co-Lipschitz} is satisfied; that is, to show that there exists $\rho=\rho(z_0)>0$ such that 
    \begin{equation}\label{f is 1/2-co-lip condition}
        |F_1(x)-F_1(z_0)|\geq \dfrac{\left|x-z_0\right|}{2} \quad \text{for each }x\in B_\rho(z_0).
    \end{equation}
Recall by Corollary~\ref{standard is strongly co-Lipschitz} that $f_n$ is strongly 1-co-Lipschitz at $z_0$. Hence there exists $\rho_1=\rho_1(z_0)>0$ such that
    \begin{equation}\label{standard is 1-strongly co-Lipschitz}
        \left|f_n(z_0)-f_n(x)\right|\geq \left|z_0-x\right|\quad\text{for each }x\in B_{\rho_1}(z_0).
    \end{equation}
Choose $\rho=\rho\left(z_0\right)>0$ sufficiently small such that $\rho<\rho_1$ and $B_{\rho}(z_0)\subseteq \C\setminus \overline{U_1}$. Let $x\in B_{\rho}(z_0)$ and put $s=|x-z_0|<\rho$. Recall \eqref{f specific form}, that is $F_1=a_0 + f_n+\sum_{k=1}^{n-1}a_kg_{k,n}$, and so
    \begin{align}
        \left|F_1(x)-F_1(z_0)\right|&=\left|\left(f_n(z_0)-f_n(x)\right)+\sum\limits_{k=1}^{n-1}a_k\left(g_{k,n}(z_0)-g_{k,n}(x)\right)\right|\nonumber\\
        &\geq \left|f_n(z_0)-f_n(x)\right|-\sum\limits_{k=1}^{n-1}|a_k|\left|g_{k,n}(z_0)-g_{k,n}(x)\right|\\
        &\geq s-\sum\limits_{k=1}^{n-1}|a_k|\left|g_{k,n}(z_0)-g_{k,n}(x)\right|\label{final inequality in the decomposition of F1},
    \end{align}
where the last inequality follows from \eqref{standard is 1-strongly co-Lipschitz}. We show
    \begin{equation} \label{sum is sufficiently small for pointwise co-Lip in claim 4}
        \sum\limits_{k=1}^{n-1}|a_k|\left|g_{k,n}(z_0)-g_{k,n}(x)\right|\leq \dfrac{s}{2}.
    \end{equation}
Combining \eqref{sum is sufficiently small for pointwise co-Lip in claim 4} with \eqref{final inequality in the decomposition of F1} implies \eqref{f is 1/2-co-lip condition} which proves $F_1$ is pointwise $\frac{1}{2}$-co-Lipschitz at $z_0$ as claimed. 

To see \eqref{sum is sufficiently small for pointwise co-Lip in claim 4} recall Notation~\ref{D_k notation and choice of R}, in particular, recall \ref{R is sufficiently large concerning lipschitz}. As $R\geq D_k$, by Lemma~\ref{k/n general function lipschitz}, $g_{k,n}$ is $1/(2n|a_k|)$-Lipschitz on $\C\setminus B_R(0)$ for those $k\in\{1,\dots,n-1\}$ where $a_k\neq 0$. Hence 
    \begin{equation*} 
        \sum\limits_{k=1}^{n-1}|a_k|\left|g_{k,n}(z_0)-g_{k,n}(x)\right|\leq \sum\limits_{k=1}^{n-1}\dfrac{|z_0-x|}{2n}=\sum\limits_{k=1}^{n-1}\dfrac{s}{2n}\leq\dfrac{s}{2}.
    \end{equation*} 
\end{innerproof}

\begin{remark} \label{remark about pointwise co-Lip outside of W}
    Recall by Claim~\ref{F_1 co-Lip in U_2 outside W} that $F_1$ is pointwise $c_0$-co-Lipschitz at each $z\in U_2\setminus W$ and by Claim~\ref{F_1 is pointwise co-Lip outside U_1 bar} that $F_1$ is pointwise $\left(1/2\right)$-co-Lipschitz at each $z\in \C\setminus \overline{U_1}$. Therefore defining $c_1:=\min\left\{c_0,\frac{1}{2}\right\}$ we conclude $c_1>0$ and
\begin{equation}\label{F_1 is locally co-Lip on C take V}
F_1\text{ is pointwise }c_1\text{-co-Lipschitz at each }z\in\C\setminus W. 
\end{equation}
\end{remark}

We continue by defining the amended homeomorphism $h_2:\C\to\C$, which coincides with $h_1$ on $\C\setminus V$, and prove the pointwise co- and Lipschitz properties of the amended function $F_2=P\circ h_2$.  Indeed, define $h_2:\C\to\C$ via 
    \begin{equation}\label{h_2 equation}
        h_2(z)=\begin{cases} 
            h_1(z),&\quad\text{if $z\not\in V$};\\
            z_j+r^{1-\frac{1}{m_j}}|z-z_j|^{1/m_j}e^{i\arg(z-z_j)},&\quad\text{if $\left|z-z_j\right|\leq r$, $z_j\in S(P')$}.
            \end{cases}
    \end{equation}
See Notation~\ref{definition of S(P')} for definition of $m_j$. To check that $h_2$ is a homeomorphism first note that $\restr{h_2}{\C\setminus \text{Int}(V)}=\restr{h_1}{\C\setminus \text{Int}(V)}$ and $\restr{h_2}{\overline{B}_r(z_j)}$ is continuous for each $z_j\in S(P')$, thus $h_2$ is continuous. Further, as $h_2(\overline{B}_r(z_j))=h_1(\overline{B}_r(z_j))=\overline{B}_r(z_j)$, both $\restr{h_2}{\overline{B}_r(z_j)}$ and $\restr{h_2}{\C\setminus \text{Int}(V)}$ are bijective, and $h_2(\C\setminus V)\cap h_2(V)=h_1(\C\setminus V)\cap h_1(V)=\emptyset$, we conclude that $h_2:\C\to\C$ is bijective. Finally as $\restr{h_2^{-1}}{\overline{B}_r(z_j)}$ is continuous for each $z_j\in S(P')$ and $\restr{h_2^{-1}}{\C\setminus\text{Int}(V)}=\restr{h_1^{-1}}{\C\setminus\text{Int}(V)}$, we conclude $h_2$ is a homeomorphism of the plane to itself. 

Recall that $P(w)=(w-z_j)^{m_j}Q_j(w)+P(z_j)$ and so $F_2(z)=P(h_2(z))$ has the following form: 
\begin{equation}\label{equation: defining the mapping F_2}
    F_2(z)=\begin{cases}
                    F_1(z),\quad&\text{if $z\not\in V$};\\
                    P(z_j)+r^{m_j-1}f_{m_j}(z-z_j)Q_j(h_2(z)),\quad&\text{if $|z-z_j|\leq r$, $z_j\in S(P')$},
                \end{cases}
\end{equation}
where $f_{m_j}$ is defined as in Lemma~\ref{archetypal case}. 

Clearly, $F_1(z)=F_2(z)$ for each $z\in\partial V$ as $\restr{h_1}{\partial B_r(z_j)}=\restr{h_2}{\partial B_r(z_j)}$ for all $z_j\in S(P')$. Moreover, since $P$ is a complex polynomial, hence an open map, and as $h_2$ is a homeomorphism, we conclude that $F_2$ is an open map. 

\begin{remark}\label{remark: if m_j=n}
    If there exists $z_j\in S(P')$ such that $m_j=n$, then $P(z)=P(z_j)+Q_j(z_j)(z-z_j)^n$ where $Q_j(z_j)\neq 0$. Therefore, $S(P')=\{z_j\}$ and so $F_2(z)=P(z_j)+Q_j(z_j)r^{n-1}f_n(z-z_j)$ for $z\in B_r(z_j)$. Hence, in such a case by Lemma~\ref{archetypal case}, $F_2$ is pointwise $\left(\left|Q_j(z_j)\right|r^{n-1}\right)$-co-Lipschitz and pointwise $\left(\left|Q_j(z_j)\right|nr^{n-1}\right)$-Lipschitz at each $z\in B_r(z_j)$. 
\end{remark}

\begin{claim}\label{F_2 lip on each ball}
   For each $z_j\in S(P')$ there exists $d_j>0$ such that $F=\restr{F_2}{\overline{B}_r(z_j)}$ is $d_j$-Lipschitz when considered as a function from $\overline{B}_r(z_j)$ to $F_2(\overline{B}_r(z_j))$. 
\end{claim} 

\begin{innerproof}
Fix $z_j\in S(P')$. We shall show that $F_2$ is pointwise $d_j$-Lipschitz at each $x\in B_r(z_j)$ for some $d_j>0$; the claim then follows by applying Lemma~\ref{pointwise Lip on convex open implies restriction is lipschitz} followed by Lemma~\ref{continuous and lip on set, then lip on closure}. 

If $m_j=n$, then by Remark~\ref{remark: if m_j=n} it follows $F_2$ is pointwise $\left(\left|Q_j(z_j)\right|nr^{n-1}\right)$-Lipschitz at each $z\in B_r(z_j)$.

Suppose that $m_j<n$. If $x=z_j$, then for each $y\in B_r(z_j)$, as $F_2(x)=F_2(z_j)=P(z_j)$ and $|f_{m_j}(y-z_j)|=|y-z_j|$, 
    \begin{equation*}
        |F_2(x)-F_2(y)|=r^{m_j-1}|Q_j(h_2(y))|\cdot|y-z_j|=r^{m_j-1}|Q_j(h_2(y))|\cdot|x-y|. 
    \end{equation*}
Since $h_2\left(B_r(z_j)\right)=B_r(z_j)$, by Claim~\ref{choice of r}~\ref{property iv of r}, $F_2$ is pointwise $2r^{m_j-1}|Q_j(z_j)|$-Lipschitz at $x=z_j$. 

Suppose now that $x\in B_r(z_j)\setminus\{z_j\}$. Let $\rho_1>0$ be such that $B_{\rho_1}(x)\subseteq B_r(z_j)$. Further, for each $l\in\{1,\dots,n-m_j\}$, let $\rho_{2,l}>0$ be given by Corollary~\ref{Phi is bounded}, where $w=x-z_j\neq 0$, so that for each $z\in B_{\rho_{2,l}}(w)$, $\Phi_{l,m_j}(z,w)$ is well-defined and 
    \begin{equation}\label{equation: boundedness of phi in F_2 Lip}
        \left|\Phi_{l,m_j}(z,w)\right|<1+\left|\Phi_{l,m_j}(w,w)\right|. 
    \end{equation}
Define $\rho_2:=\min\{\rho_{2,l}:1\leq l\leq n-m_j\}$ and $\rho:=\min(\rho_1,\rho_2)$. Note if $y\in B_{\rho}(x)$, then $z=y-z_j\in B_{\rho}(w)$. Considering \eqref{taylor expansion of Q_j about z_j}, \eqref{h_2 equation}, \eqref{equation: defining the mapping F_2} and Lemma~\ref{archetypal case} we deduce that if $y\in B_{\rho}(x)$, then 
\begin{align}
        \nonumber &F_2(y)-F_2(x)=F_2\left(z_j+|y-z_j|e^{i\arg(y-z_j)}\right)-F_2\left(z_j+|x-z_j|e^{i\arg(x-z_j)}\right)\\
        &=r^{m_j-1}\left(f_{m_j}(z)-f_{m_j}(w)\right)\left(c_{0,j}+\sum\limits_{l=1}^{n-m_j}r^{\frac{l(m_j-1)}{m_j}}c_{l,j}\cdot\Phi_{l,m_j}\left(z,w\right)\right),\label{decomposing F_2 into multiples of standard lipschit with phi function part 1}
\end{align}
where $z=y-z_j$ and $w=x-z_j$. To see that $F_2$ is pointwise Lipschitz at $x$, as $f_{m_j}$ is Lipschitz and $|z-w|=|y-x|$, it suffices to observe that $|\Phi_{l,m_j}(z,w)|$ are uniformly bounded over $z\in B_{\rho}(w)$ and $|w|=|x-z_j|<r<1$. Indeed, by \eqref{equation: boundedness of phi in F_2 Lip} as $l\in\{1,\dots,n-m_j\}$, observe that 
    \begin{equation*}
        \left|\Phi_{l,m_j}(z,w)\right|<1+|w|^{l/m_j}\dfrac{l+m_j}{m_j}\leq 1+\dfrac{nr^{1/m_j}}{m_j}\leq 1+n. 
    \end{equation*}
Hence, we conclude that there exists $d_j>0$ such that $F_2$ is pointwise $d_j$-Lipschitz at each $x\in B_r(z_j)$, which as explained above, implies the statement of Claim~\ref{F_2 lip on each ball}.
\end{innerproof}

\begin{claim}\label{claim: F_2 is L-Lipschitz on C}
There exists $L>0$ such that $F_2$ is $L$-Lipschitz on $\C$. 
\end{claim}

\begin{innerproof}
    Recall Remark~\ref{F_1 is Lipschitz on C}. Since $F_1(z)=F_2(z)$ for $z\in (\C\setminus V)\cup\partial V$ we conclude $F_2$ is pointwise $L_1$-Lipschitz at each $z\in \C\setminus V$ and, moreover, 
    \begin{equation*}
        |F_2(z)-F_2(w)|\leq L_1|z-w|\quad\text{ for $z\in\partial V$ and $w\in \C\setminus V$}. 
    \end{equation*}

 Therefore, by Claim~\ref{choice of r}~\ref{property i of r}, Claim~\ref{F_2 lip on each ball} and by defining $L$ to be the maximum of $L_1$ and $\max_{j:z_j\in S(P')}d_j$, we conclude $F_2$ is pointwise $L$-Lipschitz at each $z\in \C$. Hence Lemma~\ref{pointwise Lip on convex open implies restriction is lipschitz} implies that $F_2$ is $L$-Lipschitz on $\C$. 
\end{innerproof}

We now turn our attention to the co-Lipschitzness of $F_2$. 

\begin{claim}\label{F_2 is pointwise co-Lip in open balls}
    For each $z_j\in S(P')$ and $z\in B_r(z_j)$, the mapping $F_2$ is pointwise $\alpha_j$-co-Lipschitz at $z$, where $\alpha_j$ is defined in \eqref{definition of alpha_j for strongly co-Lip on open balls claim}.
\end{claim}

\begin{innerproof} 
Fix $z_j\in S(P')$ and define 
    \begin{equation}\label{definition of alpha_j for strongly co-Lip on open balls claim}
        \alpha_j:=\dfrac{r^{m_j-1}\left|Q_j(z_j)\right|}{2}.
    \end{equation}
If $m_j=n$, then by Remark~\ref{remark: if m_j=n} it follows that, as $\alpha_j<r^{n-1}|Q_j(z_j)|$, $F_2$ is pointwise $\alpha_j$-co-Lipschitz at each $z\in B_r(z_j)$. 

Suppose that $m_j<n$. By \eqref{Q_j(z_j) is not zero} we have that $\alpha_j>0$. To show $F_2$ is pointwise $\alpha_j$-co-Lipschitz at each $z\in B_r(z_j)$ we first show for each $z\in \overline{B}_r(z_j)$ that there exists $\rho=\rho(z)>0$ such that 
    \begin{equation} \label{strongly co-Lipschitz equation for V}
        \left|F_2(z)-F_2(y)\right|\geq \alpha_j\left|z-y\right|
    \end{equation}
for each $y\in B_{\rho}(z)\cap \overline{B}_r(z_j)$. We emphasize that \eqref{strongly co-Lipschitz equation for V} holds not only for $z\in B_r(z_j)$ but also for $z\in \partial B_r(z_j)$, and this fact is used later in the proof of Claim~\ref{claim for pointwise co-Lip on boundary}. 

Consider first when $z=z_j$. Let $\rho=r$ and $y\in B_{\rho}(z)$. From \eqref{equation: defining the mapping F_2}, we deduce that 
    \begin{align*}
        |F_2(z)-F_2(y)|=r^{m_j-1}|y-z||Q_j(h_2(y))|. 
    \end{align*}
Since $h_2(B_r(z_j))=B_r(z_j)$, by Claim~\ref{choice of r}~\ref{property iv of r}, we conclude that $F_2$ satisfies \eqref{strongly co-Lipschitz equation for V} when $z=z_j$. 

Fix $z\in \overline{B}_r(z_j)\setminus\{z_j\}$. Let $\rho_1=\rho_1(z)>0$ be defined by 
    \begin{equation}\label{definition of rho_1 for strongly co-Lip in closed balls}
        \rho_1(z)=\begin{cases}
            r,\quad&\text{if $z\in\partial B_r(z_j)$};\\
            r-|z-z_j|,\quad&\text{if $z\in B_r(z_j)\setminus\{z_j\}$}.
        \end{cases}
    \end{equation}
By Corollary~\ref{standard is strongly co-Lipschitz}, since $f_{m_j}$ is strongly 1-co-Lipschitz at $\left(z-z_j\right)\in \overline{B}_r(0)$ there exists $\rho_2=\rho_2(z)>0$ such that for any $x\in B_{\rho_2}(z-z_j)$ it follows that 
    \begin{equation}\label{strongly co-Lipschitz property for the standard map in the claim for co-Lip in balls}
        \left|f_{m_j}(x)-f_{m_j}(z-z_j)\right|\geq \left|x-\left(z-z_j\right)\right|.
    \end{equation}
Further by Corollary~\ref{Phi is bounded}, for each $l\in\{1,\dots,n-m_j\}$, let $\rho_{3,l}>0$ be such that for each $y\in B_{\rho_{3,l}}(z)$, $\Phi_{l,m_j}(y-z_j,z-z_j)$ is well-defined and 
    \begin{equation}\label{equation: Phi bound for co-Lip of F_2}
        \left|\Phi_{l,m_j}(y-z_j,z-z_j)\right|<r^{1/m_j}+|z-z_j|^{l/m_j}\dfrac{l+m_j}{m_j}. 
    \end{equation}
Define $\rho_3:=\min\left\{\rho_{3,l}:1\leq l\leq n-m_j\right\}$ and let $\rho=\rho(z)>0$ be given by $\rho=\min\left(\rho_1,\rho_2,\rho_3\right)$. We claim for $y\in B_{\rho}(z)\cap \overline{B}_r(z_j)$ that 
    \begin{equation}\label{F_2 first expression regarding result for co-Lip}
        \left|F_2(y)-F_2(z)\right|\geq \alpha_j\left|f_{m_j}(y-z_j)-f_{m_j}(z-z_j)\right|.
    \end{equation}

Fix $y\in B_{\rho}(z)\cap \overline{B}_r(z_j)$. By using $y\in \overline{B}_r(z_j)$ for $F_2(y)$, $z\neq z_j$ and $y\in B_{\rho}(z)$ for the well-definedness of $\Phi_{l,m_k}(y-z_j,z-z_j)$ and recalling \eqref{decomposing F_2 into multiples of standard lipschit with phi function part 1}, it follows that
    \begin{align*}
        |F_2&(y)-F_2(z)|\geq\\
        &r^{m_j-1} \left(|c_{0,j}|-\max\limits_{l\in\{1,\dots,n-m_j\}}\left|\Phi_{l,m_j}\left(y-z_j,z-z_j\right)\right|\cdot\sum\limits_{k=1}^{n-m_j}r^{\frac{k(m_j-1)}{m_j}}|c_{k,j}|\right)\\
        &\times\left|f_{m_j}\left(y-z_j\right)-f_{m_j}\left(z-z_j\right)\right|.
    \end{align*}
Therefore, since $r<1$, see Claim~\ref{choice of r}, to show \eqref{F_2 first expression regarding result for co-Lip} it suffices to prove, as $c_{0,j}=Q_j(z_j)$, that for all $l\in\{1,\dots,n-m_j\}$, 
    \begin{equation}\label{required property for Phi and Q_j for co-Lip property in B}
        \left|\Phi_{l,m_j}\left(y-z_j,z-z_j\right)\right|\sum\limits_{k=1}^{n-m_j}|c_{k,j}|\leq \dfrac{|Q_j(z_j)|}{2}.
    \end{equation}
This is trivial when $\sum_{k=1}^{n-m_j}|c_{k,j}|=0$. Suppose $\sum_{k=1}^{n-m_j}|c_{k,j}|\neq 0$. By property \ref{property iii of r} of Claim~\ref{choice of r}, which refers to the inequality \eqref{properties of Phi following from uniform continuity} of Corollary~\ref{Phi is bounded}, since $|y-z_j|<\rho\leq \rho_3$, $z\in \overline{B}_r(z_j)$, $m_j\geq 1$ and $l\leq n-m_j$, note that
    \begin{align*}
        \left|\Phi_{l,m_j}\left(y-z_j,z-z_j\right)\right|&<r^{1/m_j}+|z-z_j|^{l/m_j}\dfrac{l+m_j}{m_j}&\text{by \eqref{equation: Phi bound for co-Lip of F_2},}\\
        &\leq (1+n)r^{1/m_j}\\
        &\leq\dfrac{|Q_j(z_j)|}{2\sum\limits_{k=1}^{n-m_j}|c_{k,j}|}&\text{ by Claim~\ref{choice of r}~\ref{property iii of r}.}
    \end{align*}
Thus \eqref{required property for Phi and Q_j for co-Lip property in B} follows and so \eqref{F_2 first expression regarding result for co-Lip} is satisfied, as claimed. 

Since $\rho\leq \rho_2$ and $y\in B_{\rho}(z)$ it follows $\left(y-z_j\right)\in B_{\rho_2}(z-z_j)$. Therefore, by \eqref{strongly co-Lipschitz property for the standard map in the claim for co-Lip in balls},
    \begin{equation*} 
        |f_{m_j}(y-z_j)-f_{m_j}(z-z_j)|\geq \left|(y-z_j)-(z-z_j)\right|=|y-z|.
    \end{equation*}
Hence, combining this with \eqref{F_2 first expression regarding result for co-Lip} yields
    \begin{equation*} 
        |F_2(y)-F_2(z)|\geq \alpha_j|f_{m_j}(z-z_j)-f_{m_j}(y-z_j)|\geq \alpha_j|y-z|.
    \end{equation*}
Thus we deduce that for each $z\in\overline{B}_r(z_j)$ there exists $\rho>0$ such that \eqref{strongly co-Lipschitz equation for V} holds for all $y\in B_{\rho}(z)\cap \overline{B}_r(z_j)$. 

If $z\in B_r(z_j)$, by \eqref{definition of rho_1 for strongly co-Lip in closed balls} and since $\rho\leq \rho_1$ we note $B_{\rho}(z)\subseteq B_{r}(z_j)$. Hence for each $y\in B_{\rho}(z)$, \eqref{strongly co-Lipschitz equation for V} is satisfied. Therefore, since $F_2=P\circ h_2$ is an open map, by Corollary~\ref{nicer co-lip condition for open maps}, Remark~\ref{remark: strongly co-Lip in open subsets} and since $B_r(z_j)$ is open in $\C$, we conclude that $F_2$ is pointwise $\alpha_j$-co-Lipschitz at any $z\in B_r(z_j)$.
\end{innerproof}

\begin{remark}\label{claim that F_2 is locally co-Lipschitz on interior of V}
Taking $c_2:=\min_{z_j\in S(P')}\alpha_j>0$ we deduce 
    \begin{equation} \label{F_2 is locally co-Lip on interior of V}
        F_2\text{ is pointwise }c_2\text{-co-Lipschitz at each }z\in\text{Int}(V).
    \end{equation}
\end{remark}

\begin{claim} \label{claim for pointwise co-Lip on boundary}
   There exists $c_3>0$ such that $F_2:\C\to\C$ is pointwise $c_3$-co-Lipschitz at each $z\in \partial V$.
\end{claim}

\begin{innerproof} 
 Let $c_3:=\min(c_0,c_2)$, where $c_0>0$ is given by Claim~\ref{F_1 co-Lip in U_2 outside W} and $c_2>0$ is given by Remark~\ref{claim that F_2 is locally co-Lipschitz on interior of V}. Since $F_2$ is an open map, it suffices by Corollary~\ref{nicer co-lip condition for open maps} to show for each $z\in\partial V$ there exists $\rho=\rho(z)>0$ such that if $x\in B_{\rho}(z)$, then
    \begin{equation}\label{inequality for verification of strongly co-Lipschitz on the boundary}
        \left|F_2(z)-F_2(x)\right|\geq c_3\left|z-x\right|.
    \end{equation}
Fix $z\in\partial V$ and let $j$ be such that $z\in \partial B_r(z_j)$. Let $\rho_1>0$ be such that $B_{\rho_1}(z)\subseteq U_2$ and $B_{\rho_1}(z)\cap V\subseteq \overline{B}_r(z_j)$; note such $\rho_1>0$ exists by Claim~\ref{choice of r}~\ref{property i of r}. Since $\partial V\subseteq U_2\setminus \text{Int}(V)$ and $\restr{F_1}{U_2\setminus \text{Int}(V)}=\restr{F_2}{U_2\setminus \text{Int}(V)}$, by \eqref{F_1 is strongly co-Lipschitz outside of Int(V)} and $c_3\leq c_0$ there exists $\rho_2\in (0,\rho_1)$ such that \eqref{inequality for verification of strongly co-Lipschitz on the boundary} is satisfied for each $x\in B_{\rho_2}(z)\cap (U_2\setminus\text{Int}(V))=B_{\rho_2}(z)\setminus B_r(z_j)$.

Further, by \eqref{strongly co-Lipschitz equation for V} there exists $\rho\in(0,\rho_2)$ such that \eqref{inequality for verification of strongly co-Lipschitz on the boundary} is satisfied for each $x\in B_{\rho}(z)\cap \overline{B}_r(z_j)$ since $c_3\leq c_2\leq \alpha_j$; see Remark~\ref{claim that F_2 is locally co-Lipschitz on interior of V}. 

We then conclude that \eqref{inequality for verification of strongly co-Lipschitz on the boundary} is satisfied for each $x\in B_{\rho}(z)$. As $F_2$ is an open map, Corollary~\ref{nicer co-lip condition for open maps} implies the statement of Claim~\ref{claim for pointwise co-Lip on boundary}.
\end{innerproof}

\begin{claim}\label{claim: showing F_2 is c-co-Lipschitz} 
There exists $c>0$ such that $F_2$ is $c$-co-Lipschitz on $\C$. 
\end{claim}

\begin{innerproof} 
Let $c:=\min(c_1,c_2,c_3)$, where $c_1$ is given by Remark~\ref{remark about pointwise co-Lip outside of W}, $c_2$ is given by Remark~\ref{claim that F_2 is locally co-Lipschitz on interior of V} and $c_3$ is given by Claim~\ref{claim for pointwise co-Lip on boundary}. Recall by \eqref{F_1 is locally co-Lip on C take V} of Remark~\ref{remark about pointwise co-Lip outside of W} that $F_1$ is pointwise $c_1$-co-Lipschitz at each $z\in \C\setminus W$. As $F_1(z)=F_2(z)$ for $z\in \C\setminus V$ and $W\subseteq V$, we conclude 
\begin{equation}\label{F_2 is loc co-Lip outside V}
F_2\text{ is pointwise }c\text{-co-Lipschitz at each }z\in\C\setminus V. 
\end{equation}
Also, Remark~\ref{claim that F_2 is locally co-Lipschitz on interior of V} implies that 
\begin{equation}\label{F_2 is pointwise co-Lipschitz outside of the boundary of V}
    F_2\text{ is pointwise }c\text{-co-Lipschitz at each }z\in \text{Int}(V).
\end{equation}
From Claim~\ref{claim for pointwise co-Lip on boundary}, \eqref{F_2 is loc co-Lip outside V} and \eqref{F_2 is pointwise co-Lipschitz outside of the boundary of V}, we conclude that $F_2$ is pointwise $c$-co-Lipschitz at each $z\in\C$. Hence an application of Lemma~\ref{everywhere local co-Lip} implies $F_2$ is $c$-co-Lipschitz on $\C$. 
\end{innerproof}

Finally, Claims~\ref{claim: F_2 is L-Lipschitz on C} and \ref{claim: showing F_2 is c-co-Lipschitz} together imply that $f:=F_2=P\circ h_2$ is an $L$-Lipschitz and $c$-co-Lipschitz mapping of the plane. 
\end{proof}


\begin{thebibliography}{15}

\bibitem{BJLPS}
S. Bates, W. B. Johnson, J. Lindenstrauss, D. Preiss and G. Schechtman, {\it Affine approximation of Lipschitz functions and nonlinear quotients}, Geom. Funct. Anal, {\bf 9} (1999), 1092-1127. 

\bibitem{everywhere}
C. Bernardi and C. Rainaldi, {\it Everywhere surjections and related topics: Examples and counterexamples}, Matematiche, {\bf 73} (2018), 71-88. 

\bibitem{Csornyei} 
M. Cs{\"o}rnyei, {\it Can one squash the space into the plane without squashing?}, Geom. Funct. Anal., {\bf 11} (2001), 933-952. 

\bibitem{manifolds} K. Fritzshe and H. Grauert, {\it From Holomorphic Functions to Complex Manifolds}, Springer, New York, NY (2002). 

\bibitem{Gromov}
M. Gromov, {\it Metric Structures for Riemannian and Non-Riemannian Spaces}, Birkh{\" a}user, Springer, New York, NY (1997). 

\bibitem{James}
I. M. James, {\it Introduction to Uniform Spaces}, Cambridge University Press, Cambridge (1990). 


\bibitem{JLPS}
W. B. Johnson, J. Lindenstrauss, D. Preiss and G. Schechtman, {\it Uniform quotient mappings of the plane}, Michigan Math. J. {\bf 47} (2000), 15-31.

\bibitem{MalevaMathe}
G. Kun, O. Maleva and A. M{\'a}th{\'e}, {\it Metric characterisation of pure unrectifiability}, Real Anal. Exchange, {\bf 31} (2006), 195-214.

\bibitem{BLD} 
R. Luisto, {\it A Characterisation of BLD-Mappings Between Metric Spaces}, Jour. of Geom. Anal., {\bf 27} (2016), 2081-2097. 

\bibitem{OlgaPaper1} 
O. Maleva, {\it Lipschitz quotient mappings with good ratio of constants}, Mathematika, {\bf 49} (2002), 159-165. 


\bibitem{MV} O. Maleva and C. Villanueva-Segovia, {\it Best constants for Lipschitz quotient mappings in polygonal norms}, Mathematika, {\bf 67} (2021), 116-144.


\bibitem{Whyburn} 
G. T. Whyburn, {\it Topological Analysis}, Princeton University Press, Princeton (1958). 


\end{thebibliography}
\end{document}